\documentclass[12pt]{amsart}
\usepackage{amsmath,amssymb, amsfonts,amscd,psfrag,graphicx,enumitem,tikz}
\usepackage{mathrsfs}
\usetikzlibrary{patterns}
\tikzstyle{vertex}=[circle, draw, inner sep=0pt, minimum size=6pt]
\newcommand{\lebn}

\theoremstyle{plain}

\newtheorem{thm}[equation]{Theorem}
\newtheorem{question}[equation]{Problem}

\newtheorem{lem}[equation]{Lemma}

\theoremstyle{definition}
\newtheorem{prob}[equation]{Problem}

\newtheorem{rem}[equation]{Remark}

\numberwithin{equation}{section}

\newcommand{\Q}{\mathbb{Q}}

\newcommand{\Z}{\mathbb{Z}}

\newcommand{\E}{\mathbb{E}}
\renewcommand{\prob}{\mathbb{P}}

\newcommand{\A}{\mathcal{A}}
\newcommand{\NE}{\mathcal{N}}

\newcommand{\lk}{\operatorname{link}}
\newcommand{\st}{\operatorname{star}}

\newcommand{\conn}{\operatorname{conn}}
\newcommand{\vsupp}{\operatorname{vsupp}}
\newcommand{\supp}{\operatorname{supp}}

\begin{document}
\bibliographystyle{plain}

\title[Topology of random $\lowercase{d}$-clique complexes]{Topology of random $d$-clique complexes}

\author[]{Demet Taylan}
\address{Department of Mathematics, Bozok University, Turkey.}
\email{demet.taylan@bozok.edu.tr}

\date{\today}

\begin{abstract} For a simplicial complex $X$, the $d$-clique complex $\Delta_d(X)$ is the simplicial complex having all subsets of vertices whose $(d + 1)$-subsets are contained by $X$ as its faces. We prove that if $p = n^{\alpha}$, with $\alpha < \max\{\frac{-1}{k-d +1},-\frac{d+1}{\binom{k}{d}}\}$ or $\alpha > \frac{-1}{\binom{2k+2}{d}}$, then the $k$-th reduced homology group of the random $d$-clique complex $\Delta_d(G_d(n,p))$ is asymptotically almost surely vanishing, and if $\frac{-1}{t}<\alpha < \frac{-1}{t+1}$ where $t = (\frac{(d+1)(k+1)}{\binom{(d+1)(k+1)}{d+1}-(k+1)})^{-1}$, then the $(kd + d -1)$-st reduced homology group of $\Delta_d(G_d(n,p))$ is asymptotically almost surely nonvanishing. This provides a partial answer to a question posed by Eric Babson.

\end{abstract}

\maketitle

\section{Introduction}
One of the famous results in random graph theory establishes that $p=\frac{log n}{n}$ is a threshold for the connectivity of Erd\H os and R\'enyi random graphs~\cite{erdos}. A $2$-dimensional analogue of Erd\H os and R\'enyi's result was obtained by Linial-Meshulam in~\cite{linial}. Meshulam-Wallach, in~\cite{meshulam}, presented a $d$-dimensional analogue for $d\geq 3$. In particular, Linial-Meshulam-Wallach theorem yields that $p=\frac{d log n}{n}$ is the threshold for the vanishing of the $(d-1)$-st homology of the random simplicial complex $G_d(n,p)$ with coefficients in a finite abelian group. Here, the random simplicial complex $G_d(n,p)$ is a $d$-dimensional simplicial complex on $[n]$ with a full $(d-1)$-dimensional skeleton and with $d$-dimensional faces are chosen independently each with probability $p$.

The topology of clique complexes of random graphs has been studied in~\cite{babson, costa, Kahle, kahle1, kahle3} and see also~\cite{kahle2} for a survey on random simplicial complexes. The analogue $\Delta_d(X)$ has been considered in~\cite{Babson}.

In this generalization, the $d$-clique complex $\Delta_d(X)$ of a finite simplicial complex $X$ is the simplicial complex on vertex set $V(X)$ of $X$ consisting of all the subsets $F\subseteq V(X)$ with $\binom{F}{d+1} \subseteq X$. As a matter of definition, $\Delta_d(X)$ contains $X$ and the full $(d-1)$-skeleton of the simplex with vertices $V(X)$.

The following is among the problems proposed by Eric Babson in the First Research School on Commutative Algebra and Algebraic Geometry (RSCAAG):

\begin{question}\cite{Babson}\label{Question: Problem on finding threshold}
Find thresholds for $H_k(X;\Q)$ to vanish with $X\in \Delta_d(G_d(n,n^{-\alpha}))$.
\end{question}

An idea suggested by Babson for this problem is to try the techniques used for the clique complexes $\Delta_1(G_1(n,n^{-\alpha}))$ of random graphs $G_1(n,n^{-\alpha})$ and is that there may be analogues of the theorems about $\Delta_1(G_1(n,n^{-\alpha}))$ (see~\cite{Babson}). We provide here a partial answer to Problem~\ref{Question: Problem on finding threshold}. In particular, we have the following.

\begin{thm}\label{thm: vanishing results on d-clique complex of random simplicial complexes}
If $p = n^{\alpha}$ then

\begin{enumerate}[label=(\roman*)]
\item if $\alpha < \max\{\frac{-1}{k-d +1},-\frac{d+1}{\binom{k}{d}}\}$ or $\alpha > \frac{-1}{\binom{2k+2}{d}}$ then for the $k$-th reduced homology group of the $d$-clique complex $\Delta_d(G_d(n,p))$ of the random simplicial complex $G_d(n,p)$ we have asymptotically almost surely $\tilde{H}_k(\Delta_d(G_d(n,p)), \Z) = 0$,
\item and if $\frac{-1}{t}<\alpha < \frac{-1}{t+1}$ with $t = (\frac{(d+1)(k+1)}{\binom{(d+1)(k+1)}{d+1}-(k+1)})^{-1}$ then asymptotically almost surely $\tilde{H}_{(k+1)d-1} (\Delta_d(G_d(n,p)), \Z) \neq 0$ holds.
\end{enumerate}
\end{thm}

We note that when $d = 1$, Theorem~\ref{thm: vanishing results on d-clique complex of random simplicial complexes} reduces to Corollary 3.7 in~\cite{Kahle} with one difference: Kahle, in~\cite{Kahle}, improves the sufficient condition $\alpha > \frac{-1}{2k+2}$ for vanishing homology of $\Delta_1(G_1(n,p))$ to $\alpha > \frac{-1}{2k+1}$.

\section{Preliminaries}\subsection{Simplicial Complexes}\label{preliminaries: simplicial complex}
An \emph {abstract simplicial complex} $\Delta$ on a finite vertex set $V$ (or $V(\Delta)$) is a set of subsets of $V$, called \emph{faces}, satisfying the following properties:
\begin{enumerate}
\item $\{v\}\in \Delta$ for all $v\in V$.
\item If $F\in\Delta$ and $H\subseteq F$, then $H\in\Delta$.
\end{enumerate}

For a given a subset $U\subset V$, the complex $\Delta[U]:=\{\sigma\colon \sigma\in \Delta\;\textrm{,}\;\sigma\subseteq U\}$ is called the \emph{induced subcomplex} by $U$. The number of $i$-dimensional faces of a simplicial complex $\Delta$ will be denoted by $f_i(\Delta)$ and the dimension of $\Delta$ by $\dim(\Delta)$. A $d$-dimensional simplex and its boundary are denoted by $\Delta_{d+1}$ and $\partial(\Delta_{d+1})$, respectively. 

The join of two simplicial complexes $\Delta_0$ and $\Delta_1$ is denoted by 
$\Delta_0*\Delta_1$. An $n$-fold join $\underbrace{\Delta*\Delta*\dots*\Delta}_{n\; \text{times} \;\Delta}$ and $k$-dimensional skeleton of a simplicial complex $\Delta$ will be denoted by $*_{n-1} \Delta$ and $\Delta^{(k)}$, respectively.

Let $\Delta$ be a simplicial complex. For a given face $\sigma$, the \emph{link} $\lk_\Delta(\sigma)$ and the star $\st_\Delta(\sigma)$ are defined respectively by $\lk_\Delta(\sigma)=\{\tau\in\Delta\colon\tau\cap \sigma=\emptyset\;\textrm{and}\;\tau\cup\sigma\in\Delta\}$ and $\st_\Delta(\sigma) = \{\tau\in\Delta\colon\tau\cup\sigma\in\Delta\}$. For a vertex $x$ in $\Delta$, we abbreviate $\lk_\Delta(\{x\})$ and $\st_\Delta(\{x\})$ to $\lk_{\Delta}(x)$ and $\st_{\Delta}(x)$ (or simply $\lk(x)$ and $\st(x)$ if no confusion arises), respectively.

A \emph{strongly connected} simplicial complex is a pure simplicial complex in which for each pair of facets $(\sigma, \tau)$ there is a sequence of facets $\sigma = \sigma_0,\sigma_1,\dots,\sigma_{n-1} = \tau$ such that the intersection $\sigma_i\cap \sigma_{i+1}$ of any two consecutive elements in the sequence is a codimension one face of both $\sigma_i$ and $\sigma_{i+1}$. 

We refer the reader to~\cite{hatcher} for background on simplicial homology. For a $d$-chain $C$ of a simplicial complex $\Delta$, the set of all $d$-simplices appearing in $C$ with non-zero coefficients is called the~\emph{support} $\supp(C)$ of $C$. We denote the simplicial complex obtained by taking the downwards closure of $\supp(C)$ with respect to containment by $\Delta(\supp(C))$. The \emph{vertex support} $\vsupp(C)$ is the vertex set of $\Delta(\supp(C))$.

For any minimal representative $\gamma$ of a class in the reduced homology groups $\tilde{H}_{d}(\Delta;\Z)$ of a simplicial complex $\Delta$ with coefficients in $\Z$, the associated simplicial complex $\Delta(\supp(\gamma))$ is a strongly connected $d$-dimensional subcomplex of $\Delta$.

Let $\gamma$ be a nontrivial $k$-cycle in a simplicial complex $\Delta$, with minimal vertex support. Then $\gamma\cap \lk_{\Delta}(v)$ for $v\in \vsupp(\gamma)$ is defined as a $\Z$-linear combination of $(k-1)$-dimensional faces appearing in $\Delta(\supp(\gamma))\cap \lk_{\Delta}(v)$. See~\cite{Kahle} for further details.

\begin{lem}\label{lem: nontrivial cycles on strongly connected simplicial complexes} \cite{Kahle}
If $\gamma$ is a nontrivial $k$-cycle in a simplicial complex $\Delta$, with minimal vertex support, then $\gamma\cap \lk_{\Delta}(v)$ is a nontrivial $(k-1)$-cycle in $\lk_{\Delta}(v)$ for any element $v$ in the vertex support $\vsupp(\gamma)$ of $\gamma$.
\end{lem}

For $k\geq 0$, a topological space $X$ is said to be \emph{$k$-connected} if for every $i\leq k$, every continuous function $f$ from an $i$-dimensional sphere $S^i$ into $X$ is homotopic to a constant map. By convention, $(-1)$-connected means nonempty. The connectivity $\conn(X)$ of a topological space $X$ is the largest $k$ for which $X$ is $k$-connected. 

Aharoni and Berger, in~\cite{aharoni}, introduce a domination parameter $\tilde{\gamma}(\Delta)$ of a simplicial complex $\Delta$ on $V$, which is defined as the minimal size of a set $A\subseteq V$ such that $\tilde{sp}_{\Delta}(A) = V$ where $\tilde{sp}_{\Delta}(A) =\{v\in V\colon \textrm{there exists some face}\; \sigma\subseteq A\;\textrm{such that}\; \sigma\cup \{v\} \notin \Delta\}$. Call a set $A\subseteq V$ with the property that $\tilde{sp}_{\Delta}(A) = V$ a \emph{strong dominating set} of $\Delta$. The following result which relates the connectivity $\conn(\Delta)$ of a simplicial complex $\Delta$ to the parameter $\tilde{\gamma}(\Delta)$ is due to Aharoni and Berger~\cite{aharoni} and see also~\cite{meshulam2,meshulam3} for the particular case where $\Delta$ is a flag simplicial complex.

\begin{thm}\cite{aharoni}\label{thm: homotopic connectivity bound-aharoni} Let $\Delta$ be a simplicial complex on $V$. Then we have $conn(\Delta)\geq \frac{\tilde{\gamma}(\Delta)}{2}-2$.
\end{thm}

An event that depends on $n$ is said to occur~\emph{asymptotically almost surely (a.a.s.)} if the probability of the event approaches to $1$ as $n\rightarrow \infty$. 

\begin{thm}\label{thm: threshold for a subcomplex containment}\cite{Babson}
Let $\Delta$ be a simplicial complex on $V$. If $\Delta$ is $d$-lumpless (i.e. $\frac{|S|}{f_d(\Delta[S])} > \frac{|V|}{f_d(\Delta)}$ for every $\emptyset \subset S \subset V$) then $-\frac{|V|}{f_d(\Delta)}$ is a threshold for the event that $\Delta$ is a subcomplex of $\Delta_d(G_d(n,n^{\alpha}))$. If $\Delta$ is a $d$-lumpless $d$-complex then $-\frac{|V|}{f_d(\Delta)}$ is a threshold for $G_d(n,n^{\alpha})$ to contain a copy of $\Delta$.
\end{thm}

(Recall that $\alpha = a$ is called a ~\emph{threshold} for an event if it occurs a.a.s. for $\alpha >a$ and fails a.a.s. for $\alpha < a$).

\section{vanishing and nonvanishing homology}

In this section, we discuss the topology of the random $d$-clique complexes $\Delta_d(G_d(n,p))$. For comparison purposes, we keep in mind that the threshold for $G_d(n,n^{\alpha})$ to contain a copy of the $d$-skeleton of a $k$-dimensional simplex $\Delta_{k+1}^{(d)}$ is $-\frac{d+1}{\binom{k}{d}}$. More precisely, since the $d$-skeleton of a $k$-dimensional simplex is $d$-lumpless, Theorem~\ref{thm: threshold for a subcomplex containment} gives that if $p = n^{\alpha}$ with $\alpha > -\frac{d+1}{\binom{k}{d}}$ then a.a.s. $\dim(\Delta_d(G_d(n,n^{\alpha})))\geq k$, and if $\alpha < -\frac{d+1}{\binom{k}{d}}$ then a.a.s. $\dim(\Delta_d(G_d(n,n^{\alpha}))) < k$.
 
\begin{lem}\label{lem: domination number and d-clique complex}
If $p = (\frac{m \log n +\omega(n)}{n})^{\frac{1}{\binom{m}{d}}}$ and $\omega(n)\rightarrow \infty$ then a.a.s. $\tilde{\gamma}(\Delta_d(G_d(n,p))) \geq m +1$.
\end{lem}

\begin{proof}
Let $X$ be the number of strong dominating sets of $\Delta_d(G_d(n,p))$ with cardinality $m$. For any fixed $m$-subset $D$ of $[n]$, a vertex $v\in [n]$ is contained by the set $\tilde{sp}_{\Delta_d(G_d(n,p))}(D)$ if and only if there exists some $(d-1)$-dimensional face $\sigma\subseteq D$ of $\Delta_d(G_d(n,p))$ such that $\sigma\cup \{v\}$ is a minimal non-face of $\Delta_d(G_d(n,p))$. It thus follows that the probability that $v$ is contained by $\tilde{sp}_{\Delta_d(G_d(n,p))}(D)$ is $1 - p^{\binom{m-1}{d}}$ or $1 - p^{\binom{m}{d}}$ according to the condition that $v$ is contained by $D$ or not. Hence, the probability that $D$ is a strong dominating set of $\Delta_d(G_d(n,p))$ is at most $(1 - p^{\binom{m}{d}})^n$.  Therefore, for the expectation $\E(X)$ of $X$, we have \begin{align*} \E(X) &\leq \binom{n}{m}(1-p^{\binom{m}{d}})^n\\
&\leq n^{m}e^{-p^{\binom {m}{d}}n}\\
&= n^{m}e^{-m\log n -\omega(n)}\\
&= e^{-\omega(n)} = o(1),
\end{align*}
since $\omega(n)\rightarrow \infty$. Thus, $X=0$ a.a.s. and so a.a.s. $\tilde{\gamma}(\Delta_d(G_d(n,p))) \geq m +1$.
\end{proof}

\begin{thm}\label{thm: connectivity result for d-clique complex}
If $p = (\frac{(2k + 2) \log n +\omega(n)}{n})^{\frac{1}{\binom{2k+2}{d}}}$ and $\omega(n)\rightarrow \infty$ then a.a.s. the simplicial complex $\Delta_d(G_d(n,p))$ is $k$-connected.
\end{thm}

\begin{proof}Lemma~\ref{lem: domination number and d-clique complex} taken together with Lemma~\ref{thm: homotopic connectivity bound-aharoni} gives $\conn(\Delta_d(G_d(n,p))) \geq k-\frac{1}{2}$.
\end{proof}

\begin{rem}
We note that Theorem~\ref{thm: connectivity result for d-clique complex} reduces to Corollary 3.3 in~\cite{Kahle} when $d=1$.
\end{rem}

\begin{lem}\label{lem: number of d-1 faces in nontrivial cycle}
If $\gamma$ is a nontrivial $k$-cycle in the $d$-clique complex $\Delta_d (\Gamma)$ of a simplicial complex $\Gamma$, then $f_{d-1}(\Delta_d (\Gamma)(\supp(\gamma))) \geq (d+1)(k-d+1) + d +1$ holds.
\end{lem}

\begin{proof} Let $\gamma$ be a nontrivial $k$-cycle in $\Delta_d (\Gamma)$ with minimal vertex support. The assertion is true in the case $k = d-1$. Indeed, if $\sigma$ is a $(d-1)$-dimensional face in the support of $\gamma$, then each $(d-2)$-dimensional face of $\sigma$ must be contained by a distinct $(d-1)$-dimensional face different from $\sigma$ in the support of $\gamma$, since otherwise the coefficient of $\sigma$ would be $0$ in $\gamma$.

We assume now that $k \geq d$ and apply induction. Suppose to the contrary that $f_{d-1}(\Delta_d (\Gamma)(\supp(\gamma))) \leq (d+1)(k-d+1) + d$ holds. If $v\in \vsupp(\gamma)$ then $\gamma\cap \lk(v)$ is a nontrivial $(k-1)$-cycle in $\lk(v)$ by Lemma~\ref{lem: nontrivial cycles on strongly connected simplicial complexes}. It follows by induction hypothesis that $f_{d-1}(\Delta_d (\lk(v))(\supp(\gamma\cap \lk(v)))) \geq (d+1)(k-d) + d +1$. We note that the number of $(d-1)$-dimensional faces in $\Delta_d (\Gamma)(\supp(\gamma))$ belonging to $\st(v)$ but not to $\lk(v)$ is at least $\binom{k}{d-1}$. Note then that we must have $d = 1$ or $k = d$ so that $f_{d-1}(\Delta_d (\lk(v))(\supp(\gamma\cap \lk(v)))) = (d+1)(k-d) + d + 1$ and $f_{d-1}(\Delta_d (\Gamma)(\supp(\gamma))) = (d+1)(k-d) + 2d + 1$ hold. If $d = 1$, then $\Delta_d (\Gamma)(\supp(\gamma))$ is a $2k$-dimensional simplex, a contradiction (see~\cite{Kahle} for details). If $d = k$, then  $f_{d-1}(\Delta_d (\lk(v))(\supp(\gamma\cap \lk(v)))) = d + 1$ and $f_{d-1}(\Delta_d (\Gamma)(\supp(\gamma))) = 2d + 1$ hold, which is impossible. This completes the proof.

\end{proof}

\begin{rem}
We note that, in the case of a flag simplicial complex $\Delta$, Lemma~\ref{lem: number of d-1 faces in nontrivial cycle} reduces to the well-known fact that any representative of a class in $\tilde{H}_{k}(\Delta;\Z)$ is supported on at least $2k + 2$ vertices. See Lemma 5.3 in~\cite{Kahle}.
\end{rem}

\begin{lem}\label{lem: homology is generated by cycles on small vertex sets}
If $\alpha < \frac{-1}{\binom{k}{d}}$ and $0 < \frac{k}{\binom{k}{d}N} < \frac{-1}{\binom{k}{d}}- \alpha$, then the vertex support of any strongly connected $k$-dimensional subcomplex of the $d$-clique complex $\Delta_d (G_d(n,p))$ of $G_d(n,p)$ a.a.s. has at most $N + k$ vertices, where $p = n^{\alpha}$.
\end{lem}

\begin{proof}
Let $\Delta$ be a strongly connected $k$-dimensional subcomplex of $\Delta_d (G_d(n,p))$. Let the vertices of $\Delta$ are ordered as $v_1,v_2,\dots,v_n$ so that the first $k + 1$ vertices $v_1,v_2,\dots,v_{k + 1}$ forms a $k$-face and for any other vertex $v_i$ there is at least $k$ vertices $v_j$ such that $\{v_i,v_j\}\in\Delta$ where $j < i$ (See \cite{Kahle} for more details). With this ordering, suppose to the contrary that $\Delta$ has $N + k + 1$ vertices (Here $k + 1$ is the number of vertices in a $k$-dimensional face and $N$ is the number of vertices get added in total). It then follows that the number of $d$-dimensional faces in $\Delta$ is at least $\binom{k + 1}{d + 1} + \binom{k}{d}N$. Since the $d$-skeleton of any subcomplex of $\Delta_d (G_d(n,p))$ is also a subcomplex of $G_d(n,p)$, we have 
\begin{align*} \prob (C_{\Delta}) &\leq (N + k + 1)! \binom{n}{N + k + 1} p^{\binom{k+1}{d+1} + N\binom{k}{d}} \\
&=  (N + k + 1)! \binom{n}{N + k + 1} n^{\alpha(\binom{k+1}{d+1} + N\binom{k}{d})}
\end{align*}
for the total probability, where $C_{\Delta}$ denotes the event that $\Delta_d(G_d(n,p))$ contains a simplicial complex isomorphic to $\Delta$. By the assumption $0 < \frac{k}{\binom{k}{d}N} < \frac{-1}{\binom{k}{d}}- \alpha$, we can choose $N$ and $\epsilon$ such that $\frac{k}{\binom{k}{d}N} < \epsilon < \frac{-1}{\binom{k}{d}}- \alpha$ holds. It then follows that $p = n^\alpha < n^{-(\frac{1}{\binom{k}{d}} + \epsilon )}$ and $k < \binom{k}{d}N\epsilon$. We therefore get that 
\begin{align*} \prob (C_{\Delta}) &\leq (N + k + 1)! \binom{n}{N + k + 1} n^{\alpha(\binom{k+1}{d+1} + N\binom{k}{d})} \\
&<  (N + k + 1)! \binom{n}{N + k + 1} n^{{-(\frac{1}{\binom{k}{d}} + \epsilon )}(\binom{k+1}{d+1} + N\binom{k}{d})} \\
& \leq n^{N + k + 1}  n^{{-\frac{1}{\binom{k}{d}}}(\binom{k+1}{d+1} + N\binom{k}{d})} n^{-\epsilon (\binom{k+1}{d+1} + N\binom{k}{d})} \\
& =  n^{N + k + 1}  n^{-\frac{1}{\binom{k}{d}} \binom{k+1}{d+1} - N} n^{-\epsilon \binom{k+1}{d+1}  -\epsilon N \binom{k}{d}} \\ 
& =  n^{k + 1}  n^{-\frac{k+1}{d+1}} n^{-\epsilon \binom{k+1}{d+1}  -\epsilon N \binom{k}{d}} \\
& < n^{1-\frac{k+1}{d+1}-\epsilon \binom{k+1}{d+1}}\\
& \leq  n^{-\epsilon}\\
& = O(n^{-\epsilon}) = o(1),
\end{align*}
since $k\geq d$. This, taken together with the facts that the number of non-isomorphic strongly connected $k$-dimensional simplicial complexes on $N + k + 1$ vertices is finite and any strongly connected $k$-dimensional simplicial complex on more than $N + k + 1$ vertices contains a strongly connected $k$-dimensional simplicial complex on $N + k + 1$, implies that asymptotically almost surely the vertex support of every strongly connected $k$-dimensional subcomplex of the $d$-clique complex $\Delta_d (G_d(n,p))$ of $G_d(n,p)$ has at most $N + k$ vertices. This completes the proof.
\end{proof}

\begin{thm}\label{thm: a vanishing homology result}
If $p = n ^{\alpha}$ with $\alpha < \frac{-1}{k-d +1}$ then a.a.s. $\tilde{H}_k(\Delta_d(G_d(n,p)), \Z) = 0$ holds.
\end{thm}

\begin{proof}
Let $\gamma$ be a nontrivial $k$-cycle in $\Delta_d(G_d(n,p))$ with minimal vertex support. Then $\gamma\cap \lk(v)$ is a nontrivial $(k-1)$-cycle in $\lk(v)$ for any $v\in \vsupp(\gamma)$ by Lemma~\ref{lem: nontrivial cycles on strongly connected simplicial complexes}. Therefore, we have that $f_{d-1}(\Delta_d (\lk(v))(\supp(\gamma\cap \lk(v)))) \geq (d+1)(k-d) + d +1$ for any $v\in \vsupp(\gamma)$.

Consider an arbitrary simplicial complex $\Delta$ on $m$ vertices in which the number of $(d-1)$-dimensional faces in $\lk(v)$ for any vertex $v\in V(\Delta)$ is at least $(d+1)(k-d) + d +1$. Note then that the number of $d$-dimensional faces $f_d(\Delta)$ of $\Delta$ is at least 
\begin{align*} \frac{m((d+1)(k-d +1))}{\binom{d + 1}{d}} = m(k-d +1).
\end{align*}
It then follows that the probability that $\Delta$ is a subcomplex of $\Delta_d(G_d(n,p))$ is at most 
\begin{align*} m!\binom{n}{m}p^{m(k-d +1)} &\leq n^m n^{\alpha m(k-d +1)}\\
& = n^{m ( 1 + \alpha (k-d +1))}\\
& = o(1), 
\end{align*}
since $\alpha (k-d +1) < -1$. Note that $\frac{-1}{k-d +1}\leq \frac{-1}{\binom{k}{d}}$ whenever $d\leq k$. We therefore have a.a.s. no $k$-dimensional cycles on more than $N + k + 1$ vertices by Lemma~\ref{lem: homology is generated by cycles on small vertex sets}. We also note that the number of non-isomorphic simplicial complexes $\Delta$ on $N + k$ vertices in which the number of $(d-1)$-dimensional faces in $\lk(v)$ for any vertex $v\in V(\Delta)$ is at least $(d+1)(k - d + 1)$ is finite.  It thus follows that there are asymptotically almost surely no vertex minimal nontrivial $k$-dimensional cycles in the $d$-clique complex $\Delta_d(G_d(n,p))$ and so a.a.s. $\tilde{H}_k(\Delta_d(G_d(n,p)), \Z) = 0$ holds.
\end{proof}

\begin{rem}Lemma~\ref{lem: number of d-1 faces in nontrivial cycle}, Lemma~\ref{lem: homology is generated by cycles on small vertex sets}, Theorem~\ref{thm: a vanishing homology result} generalize Lemma 5.3, Lemma 5.1, Theorem 3.6 in \cite{Kahle}, respectively.
\end{rem}

Lemma~\ref{d-skeleton of joins of boundaries of d-simplexes} provides us with an example of a lumpless simplicial complex.

\begin{lem}\label{d-skeleton of joins of boundaries of d-simplexes} For $d,k\geq 1$, the $d$-skeleton of the $(k+1)$-fold join $K:= *_k\partial(\Delta_{d+1})$ of boundaries of $d$-dimensional simplexes is $d$-lumpless, i.e. $\frac{f_0(K^{(d)}[S])}{f_d(K^{(d)}[S])}> \frac{f_0(K^{(d)})}{f_d(K^{(d)})}$ for every subset $\emptyset\neq S\subset V(K)$.
\end{lem}

\begin{proof}

Suppose that $S$ is a non-empty subset of the vertex set $V(K)$ of $K$. The claim is obviously true if $|S| = 1$. Assume now that $2\leq |S| = (d+1)(k+1)-n$, where $k + 1 > n \geq 1$. Clearly, we have
\begin{align*}
\frac{f_0(K^{(d)}[S])}{f_d(K^{(d)}[S])}\geq \frac{(d+1)(k+1)-n}{\binom{(d+1)(k+1)-n}{d+1} -(k+1-n)}.
\end{align*}
Thus we need only show that
\begin{align*} \frac{(d+1)(k+1)-n}{\binom{(d+1)(k+1)-n}{d+1} -(k+1-n)} > \frac{f_0(K^{(d)})}{f_d(K^{(d)})}
\end{align*}
holds. Note that
\begin{align*}&\frac{(d+1)(k+1)-n}{\binom{(d+1)(k+1)-n}{d+1} -(k+1-n)}\\
&\qquad =\frac{(d+1)!}{\underbrace{(dk + d + k - n)(dk + d + k - n -1) \dots(dk + k - n + 1 )}_{d\;\text{factors}} - d! + \frac{ndd!}{(d+1)(k+1)-n}}\\
&\qquad\geq \frac{(d+1)!}{(dk + d + k - n)(dk + d + k - n -1)\dots(dk + k - n + 1 ) - d! + \frac{ndd!}{2}}.
\end{align*}
This last line holds, since $(d+1)(k+1)-n\geq 2$ by our assumption on $S$. We then have that
\begin{align*}&\frac{(d+1)!}{(dk + d + k - n)(dk + d + k - n -1) \dots(dk + k - n + 1 ) - d! + \frac{ndd!}{(d+1)(k+1)-n}} \\
&\qquad > \frac{(d+1)!}{(dk + d + k - n)(dk + d + k - n -1)\dots(dk + k - n + 1 ) - d! + \frac{n(d+1)!}{2}}.
\end{align*}
For the expression
\begin{align*}&(dk + d + k - n)(dk + d + k - n -1)\dots(dk + k - n + 1 )
\end{align*}
in the denominator, we have
\begin{align*}&(dk + d + k - n)(dk + d + k - n -1)\dots(dk + k - n + 1 )  \\&\quad= (dk + d + k)(dk + d + k - n -1) \dots(dk + k - n + 1 ) \\ &\qquad- n (dk + d + k - n -1) \dots(dk + k - n + 1 ).
\end{align*}
Consider now the expression $-n(dk + d + k - n -1) \dots(dk + k - n + 1)$.
Note that there are $d-1$ factors in the product $(dk + d + k - n -1) \dots(dk + k - n + 1 )$ and the occurence of an $i$.th factor in the product $(dk + d + k - n -1) \dots(dk + k - n + 1 )$ yields that $d\geq i +1$. We next rewrite the expression
\begin{align*}&-n (dk + d + k - n -1)(dk + d + k - n -2)\\
&\qquad (dk + d + k - n -3) \dots(dk  + d + k - n -(d-1))
\end{align*}
as
\begin{align*}&\underbrace{-n}(\underbrace{(d +1)} +dk + k - n -2)\\
&\qquad (\underbrace{d} +dk + k - n -2)(\underbrace{(d-1)} + dk + k - n -2) \dots(\underbrace{3} + dk + k - n -2).
\end{align*}
Clearly, it contains the term $\frac{-(n)(d+1)!}{2}$ and the expression $dk + k - n -2$ is contained by every factor in the product. If $d\geq 2$ then $dk + k - n -2 \geq 0$ holds and thus we have
\begin{align*}&\frac{(d+1)!}{(dk + d + k - n)(dk + d + k - n -1)\dots(dk + k - n + 1 ) - d! + \frac{n(d+1)!}{2}}\\ 
&\qquad \geq \frac{(d+1)!}{(dk + d + k)(dk + d + k - n -1)\dots(dk + k - n + 1 ) - d!}.
\end{align*}
Since
\begin{align*}&\frac{(d+1)!}{(dk + d + k)(dk + d + k - n -1)\dots(dk + k - n + 1 ) - d!}\\
&\qquad > \frac{(d+1)!}{(dk + d + k)(dk + d + k - 1)\dots(dk + k + 1 ) - d!},
\end{align*}
it follows that
\begin{align*}&\frac{(d+1)!}{(dk + d + k - n)(dk + d + k - n -1)\dots(dk + k - n + 1 ) - d! + \frac{n(d+1)!}{2}}\\
&\qquad > \frac{(d+1)!}{(dk + d + k)(dk + d + k - 1)\dots(dk + k + 1 ) - d!}\\
&\qquad = \frac{f_0(K^{(d)})}{f_d(K^{(d)})}.
\end{align*}
If $d =1$, then we obviously have that
\begin{align*}\frac{2}{2k + 1 - n - 1 + n} = \frac{f_0(K^{(1)})}{f_1(K^{(1)})}.
\end{align*}

Suppose now that $2\leq |S| = (d+1)(k+1)-n$, where $n\geq k+1$. Then
\begin{align*}&\frac{f_0(K^{(d)}[S])}{f_d(K^{(d)}[S])}\geq \frac{(d+1)(k+1)-n}{\binom{(d+1)(k+1)-n}{d+1}}
\end{align*} 
holds. We have that
\begin{align*}&\frac{(d+1)(k+1)-n}{\binom{(d+1)(k+1)-n}{d+1}}\\
&\qquad = \frac{(d+1)!}{\underbrace{(dk + d + k - n)(dk + d + k - n -1) \dots(dk + d + k -d - n +1)}_{d\;\text{factors}}- d! + d!}.
\end{align*}
It then follows that
\begin{align*}&\frac{(d+1)(k+1)-n}{\binom{(d+1)(k+1)-n}{d+1}} \geq \frac{(d+1)!}{(dk + d - 1)(dk + d -2) \dots(dk + d -d)- d! + d!},
\end{align*}
since $n\geq k+1$. The expression $(dk + d - 1)(dk + d -2) \dots(dk + d -d)$ is equal to $(dk + d)(dk + d -2) \dots(dk + d -d) - (dk + d -2)(dk + d -3) \dots(dk + d -d)$. Note that the expression
\begin{align*}&- (dk + \underbrace{d} -2)(dk + \underbrace{(d-1)} -2) (dk + \underbrace{(d-2)} -2)\dots(dk + \underbrace{(2)} -2)
\end{align*}
contains the term $-(d!)$ and $dk-2\geq 0$ whenever $d\geq 2$. We therefore get that \begin{align*}&\frac{(d+1)(k+1)-n}{\binom{(d+1)(k+1)-n}{d+1}}\\
&\qquad  \geq \frac{(d+1)!}{(dk + d - 1)(dk + d -2) \dots(dk + d -d)- d! + d!} \\
&\qquad \geq \frac{(d+1)!}{(dk + d)(dk + d -2)(dk + d -3) \dots(dk + d -d)- d!}  \\
&\qquad > \frac{(d+1)!}{(dk + d + k)(dk + d -2 + (k+1))(dk + d -3 + (k+1)) \dots(dk + (k+1))- d!} \\
&\qquad  = \frac{f_0(K^{(d)})}{f_d(K^{(d)})}.
\end{align*}
This completes the proof.

\end{proof}

\begin{rem} Recall that a $k$-dimensional octahedral sphere is a $(k+1)$-fold join of two isolated points. It is well-known in random graph theory that the $1$-skeleton of a $k$-dimensional octahedral sphere is a strictly balanced graph and $n^{\frac{-1}{k}}$ is a sharp threshold function for the random graph $G_1(n,p)$ to contain the $1$-skeleton of a $k$-dimensional octahedral sphere. We remark that Lemma~\ref{d-skeleton of joins of boundaries of d-simplexes}, taken together with Theorem~\ref{thm: threshold for a subcomplex containment} reduces to this fact when $d = 1$. 
\end{rem}

It was shown by Kahle that if $p^{k} n \rightarrow  \infty$ and $p^{k+1} n \rightarrow 0$ as $n\rightarrow \infty$ then $\Delta_1(G_1(n,p))$ a.a.s. retracts onto a sphere $S^k$ and so $\Delta_1(G_1(n,p))$ a.a.s. has nonvanishing integer $k$-th homology (see Theorem 3.5. in~\cite{Kahle}). We generalize this argument by Theorem~\ref{thm: vanishing results on d-clique complex of random simplicial complexes} (ii) so that it applies to the random $d$-clique complexes:

\begin{proof}[{\bf The proof of Theorem \ref{thm: vanishing results on d-clique complex of random simplicial complexes}}] Claim (i) is immediate from Theorem~\ref{thm: connectivity result for d-clique complex} and Theorem~\ref{thm: a vanishing homology result} by taking into account the threshold for the dimension of the random $d$-clique complex $\Delta_d(G_d(n,n^{\alpha}))$. To prove (ii), consider the $(k+1)$-fold join $K:= *_k\partial(\Delta_{d+1})$ of boundaries of $d$-dimensional simplexes. Recall that $n^{-\frac{f_0(K^{(d)})}{f_d(K^{(d)})}} = n^{-\frac{(d+1)(k+1)}{\binom{(d+1)(k+1)}{d+1}-(k+1)}}$ is a threshold function for $G_d(n,p)$ containing the $d$-skeleton of the complex $K$ as a subcomplex by Lemma~\ref{thm: threshold for a subcomplex containment} together with Lemma~\ref{d-skeleton of joins of boundaries of d-simplexes}; i.e. if $\alpha > -\frac{f_0(K^{(d)})}{f_d(K^{(d)})}$ then $G_d(n,p)$ a.a.s. contains the $d$-skeleton of $K= *_k\partial(\Delta_{d+1})$ as a subcomplex, and if $\alpha < -\frac{f_0(K^{(d)})}{f_d(K^{(d)})}$ then $G_d(n,p)$ a.a.s. does not contain $K= *_k\partial(\Delta_{d+1})$. By the assumption $\alpha > \frac{-1}{t}$, we conclude that $G_d(n,p)$ a.a.s. contains the $d$-skeleton of $K= *_k\partial(\Delta_{d+1})$. Let us choose a $(d-1)$-dimensional face $F_m$ from each of the factor $\partial(\Delta^m_{d+1})$ of the $(k+1)$-fold join $K= *_k\partial(\Delta_{d+1})$, where $1\leq m\leq k+1$. Set
\begin{align*}&\A =\{F_m\colon 1\leq m \leq k + 1\},\end{align*}
\begin{align*}&S_\A= \bigcup_{m\in [1,k+1]} F_m
\end{align*}
and
\begin{align*}&\NE(\A): = \{x\in \Delta_d(G_d(n,p))\colon F\cup \{x\} \in \Delta_d(G_d(n,p))\;\textrm{for each}\;d-\textrm{subset}\; F\subseteq S_\A\}.
\end{align*}
It then follows that the conditional probability that $\NE(\A)\neq \emptyset$ for $\A$ is no more than
 \begin{align*} p^{\binom{(k+1)d}{d}}(n-(k+1)(d+1)) + p (k+1) \leq p^{t+1} (n-(k+1)(d+1)) + p (k+1) = o(1).
\end{align*}
So a.a.s. $G_d(n,p)$ contains the $d$-skeleton of $K= *_k\partial(\Delta_{d+1})$ in which $\NE(\A) = \emptyset$. Note that in this case this subcomplex is indeed an induced subcomplex of $G_d(n,p)$, since we must have that $\{x^m_{d+1}\} \cup F_m \notin G_d(n,p)$ for any choice of $m$, where $x^m_{d+1} \in V(\Delta^m_{d+1})\setminus F_m$ and $F_m\in \A$. Note also that $\Delta_d(K)$ is a subcomplex of $\Delta_d(G_d(n,p))$ and  $\Delta_d(K)$ is homeomorphic to $S^{(k+1)d -1}$.
\end{proof}

\begin{rem}
In the proof of Theorem~\ref{thm: vanishing results on d-clique complex of random simplicial complexes} (ii), we have used the inequality $t + 1\leq \binom{d(k+1)}{d}$. To observe this, it is enough to see that $\binom{(d+1)(k+1)-1}{d} + d - (d+1) \binom{(k+1)d}{d}\leq 0$. Let $K:= *_k\partial(\Delta_{d+1})$ be the $(k+1)$-fold join of boundaries of $d$-dimensional simplexes and let $\{F_{mi}\colon 1\leq i\leq d+1\}$ denote the set of all $(d-1)$-dimensional faces in the factor $\partial(\Delta^m_{d+1})$ of the $(k+1)$-fold join $K$, where $1\leq m\leq k+1$. Set $X_i= F_{11}\cup \bigcup_{l=2}^{k+1} F_{li}$ with $1\leq i\leq d+1$. Let $K'$ denote the subcomplex of $K$ obtained by removing the vertex $x$ from $\partial(\Delta^1_{d+1})$ with $x\notin F_{11}$. Note that $\binom{d(k+1)}{d}$ counts the number of $(d-1)$-dimensional faces in each $X_i$ and thus $(d+1) \binom{d(k+1)}{d}$ counts the number of $(d-1)$-dimensional faces in $K'$ with some repetations-in particular the face $F_{11}$ is repeated $(d+1)$ times. On the other hand, $\binom{(d+1)(k+1)-1}{d}$ counts the number of $(d-1)$-dimensional faces in $K'$. We therefore have that $\binom{(d+1)(k+1)-1}{d} \leq (d+1) \binom{(k+1)d}{d} - d$.
\end{rem}

\section*{Acknowledgement}
I would like to thank all the contributers of the ``First Research School on Commutative Algebra and Algebraic Geometry'' (RSCAAG) organized by the Institute for Advanced Studies in Basic Sciences (IASBS) and the Institute for Research in Fundamental Sciences (IPM).

I would like to express my particular thanks to Eric Babson through whom I learnt about  random simplicial complexes for his amazing lectures in the research school  RSCAAG, as well as for the research sessions and the problems, among one of all is the subject of this note, for his time to look over my draft and for his invaluable comments.

\end{document}